\documentclass[12pt]{amsart}
\usepackage[english]{babel}
\usepackage{amsmath,mathrsfs,amssymb,extarrows}
\usepackage{amsthm}
\usepackage{hyperref}
\usepackage{url}
\usepackage{tikz,enumerate}
\usepackage{diagbox}
\usepackage{appendix}
\usepackage{epic}
\usepackage{tabularx}
\usepackage{lscape}
\usepackage{longtable}
\usepackage{amsfonts,amsthm,cite,color}
\usepackage{epsfig}

\numberwithin{equation}{section}

\usepackage{amsfonts,amsthm,cite,color}

\newtheorem{theorem}{Theorem}
\newtheorem{corollary}[theorem]{Corollary}

\newtheorem{lemma}[theorem]{Lemma}

\numberwithin{equation}{section}
\numberwithin{theorem}{section}
\numberwithin{defn}{section}
\allowdisplaybreaks

\begin{document}

\begin{center}
{\large {\bf Asymmetric bilateral Bailey pairs and Rogers--Ramanujan type identities}}

\vskip 6mm

{\small Xiangxin Liu$^{a}$ and Lisa Hui Sun$^{b*}$}\\[2mm]

{\small Center for Combinatorics, LPMC, Nankai University, Tianjin 300071, P.R. China} \\[2mm]

{$^{a}$liuxx@mail.nankai.edu.cn, $^{b}$sunhui@nankai.edu.cn}

\end{center}

{\noindent \bf Abstract.} 
The theory of Bailey's transform provides a systematic method for deriving $q$-identities, the key factor of which is the Bailey pair. The concept of Bailey pair was first extended to bilateral version by Paule. In this paper, following Rogers' work on Fourier series, we derive two asymmetric bilateral Bailey pairs. By inserting them into the bilateral Bailey chains, we obtain several identities of Rogers--Ramanujan type, Andrews--Gordon type and also identities on false theta functions. Furthermore, based on the Bailey lattice due to Dousse, Jouhet and Konan, we get an asymmetric bilateral Bailey lemma which leads to identities on Appell--Lerch series. Moreover, by using the asymmetric bilateral Bailey lemmas due to  Andrews and Warnaar, we get some identities on false theta functions and the generalized Hecke--type series.



\allowdisplaybreaks

\section{Introduction}

The well-known Rogers--Ramanujan identities can be stated as follows
\begin{align*}
\sum_{n=0}^{\infty}\frac{q^{n^2+(1-i)n}}{(q; q)_n}=\frac{1}{(q^{2-i}, q^{3+i}; q^5)_\infty},
\end{align*}
where $i=0,1$, which were first discovered and proved in 1894 by Rogers \cite{Rog94}, and later independently rediscovered by Ramanujan \cite{RR19}.
Since then many series--product identities similar to the form of the  Rogers--Ramanujan identities have been continuously studied and thereby are called Rogers--Ramanujan type identities.

Inspired by Rogers' proof of the Rogers--Ramanujan identities \cite{Rog17}, Bailey derived the mechanism which is named Bailey's transform in \cite{Bai49}. The key of Bailey's transform is the notion of Bailey pair, which is given as a pair of sequence of rational functions $(\alpha_n(a,q), \beta_n(a,q))_{n\geq 0}$  with respect to $a$ such that
\begin{align}\label{Baileyp}
\beta_n(a,q)=\sum_{j=0}^n \frac{\alpha_j(a,q)}{(q;q)_{n-j}(aq;q)_{n+j}}.
\end{align}
For a Bailey pair in the form of  \eqref{Baileyp}, Bailey \cite{Bai49} provided  the following fundamental result, which is a consequence of Bailey's transform.

\begin{lemma}[Bailey's lemma] \label{baileylem} If $(\alpha_n(a,q),\ \beta_n(a,q))$ forms a Bailey pair, then
\begin{align*}
&\frac{1}{(aq/\rho_1,aq/\rho_2;q)_n}\sum_{j=0}^n \frac{(\rho_1,\rho_2;q)_j(aq/\rho_1\rho_2;q)_{n-j}}{(q;q)_{n-j}} \bigg(\frac{aq}{\rho_1\rho_2}\bigg)^j \beta_j(a,q)\\
&\qquad =\sum_{j=0}^n \frac{(\rho_1,\rho_2;q)_j}{(q;q)_{n-j}(aq;q)_{n+j}(aq/\rho_1,aq/\rho_2;q)_j} \bigg(\frac{aq}{\rho_1\rho_2}\bigg)^j \alpha_j(a,q).
\end{align*}
\end{lemma}

Rogers--Ramanujan type identities are widely studied in mathematics and physics, for which the close connections with Lie algebras and statistical mechanics have improved and extended the study of Bailey's lemma, see for examples \cite{ABF84,BMS96,ASW99}. 
MacMahon \cite{MacMahon} and Schur \cite{Schur17} also gave combinatorial interpretations for the Rogers--Ramanujan  identities.  Then Gordon \cite{Gordon61} further generalized the combinatorial results for Rogers--Ramanujan  identities. A decade later, Andrews \cite{A74} presented the following $q$-formulation for Gordon's results
\begin{align}\label{eq:AGri}
\sum_{n_1\geq\cdots\geq n_{k-1}\geq0} \frac{q^{n_1^2 +\cdots +n_{k-1}^2 +n_{i}+n_{i+1}+ \cdots+n_{k-1}}}{(q)_{n_1-n_2} \cdots(q)_{n_{k-2}-n_{k-1}} (q)_{n_{k-1}}}
=\frac{(q^{i}, q^{2k-i+1}, q^{2k+1}; q^{2k+1})_\infty}{(q)_\infty},
\end{align}
with $k \geq 2$ and $1 \leq i \leq k$ being two integers, which is  the well--known Andrews--Gordon identity.
In 2010,   Andrews \cite[(1.9)]{And10} obtained that for $k\geq i\geq 2$, $k$ odd and $i$ even, then
\begin{align}\label{eq:AGev}
&\sum_{n_1 \geq \cdots \geq n_{k-1} \geq 0} \frac{q^{n_1^2+ \cdots + n_{k-1}^2 +n_1 -n_2 +n_3-n_4+\cdots+ n_{i-3}-n_{i-2} +n_{i-1} +n_i + \cdots + n_{k-1}}}{(q^2;q^2)_{n_1-n_2} \cdots (q^2;q^2)_{n_{k-2} -n_{k-1}} (q^2;q^2)_{n_{k-1}}}\nonumber\\
&\qquad\qquad\qquad\qquad
=\frac{(-q^2;q^2)_\infty (q^{i}, q^{2k+2-i}, q^{2k+2}; q^{2k+2})_\infty} {(q^2;q^2)_\infty}.
\end{align}
Recently, Andrews \cite{And12,And19} applied  the classical orthogonal polynomials and the $q$-orthogonal polynomials to study Rogers--Ramanujan type identities. In \cite{And19}, Andrews  derived the following identities
\begin{align}
&\prod_{j=1}^n (1+2xq^j+q^{2j})=\sum_{j=0}^n {2n+1\brack n-j}q^{j+1\choose 2} V_j(x) ,\label{Andrewidentity1} \\
&\prod_{j=0}^{n-1} (1+2xq^j+q^{2j})=\sum_{j=0}^n q^{j\choose 2} (1-q^{2j+1})\frac{(q)_{2n}}{(q)_{n-j}(q)_{n+j+1}} W_j(x),\label{Andrewidentity2}
\end{align}
where $x=\cos\theta$,   $V_n(x)$ and $W_n(x)$ are  Chebyshev polynomials of the third and forth kinds, respectively.
Inspired by Andrews' work, Sun \cite{Sun23} and Yao \cite{Yao25} further studied Rogers--Ramanujan type identities by constructing Bailey pairs involving Chebyshev polynomials.

As an extension of Bailey pairs, Paule  \cite{Pau87} had introduced the idea of bilateral Bailey pairs in 1987. Later, Berkovich, McCoy and Schilling \cite{BMS96} defined the full version of  bilateral Bailey pair $(\alpha_n(a,q),\beta_n(a,q))$ related to $a$ for $n\in \mathbb{Z}$, which is a pair of sequences of rational functions such that
\begin{align}\label{bbp}
\beta_n(a,q)=\sum_{j=-\infty}^{n}\frac{\alpha_j(a,q)}{(q)_{n-j}(aq)_{n+j}}.
\end{align}

Then in 2007, Andrews and Warnaar \cite{AW07} gave the symmetric and asymmetric bilateral Bailey transforms.
In 2009, Guo, Jouhet and Zeng \cite{GZJ09} studied finite forms of the Rogers--Ramanujan identities by using the bilateral Bailey lemmas. Based on these works, Chu and Zhang \cite{Chu09,Chu10} obtained a generalized bilateral Bailey lemma with $a=1$, which leads to many identities of Rogers--Ramanujan type.

Rogers provided his second proof of  Rogers--Ramanujan identities in \cite{Rog17}, which was couched in
terms of Fourier series. In this paper, following Rogers' way, we obtain two asymmetric bilateral Bailey pairs
\begin{align}
(\alpha_n(1,q),\ \beta_n(1,q))=\Big(z^{n}q^{n+1\choose 2},\ \frac{(-1/z)_{n+1}(-zq)_n}{(q)_{2n+1}}\Big),\label{asymmetricbbp}
\end{align}
and
\begin{align}
(\alpha_n(1,q),\ \beta_n(1,q))=
\Big(\frac{z^{-n} q^{\binom{n}{2}} (1-q^{2n+1})}{1-z},\ \frac{(-1/z,-z)_n}{(q)_{2n}}\Big).
\end{align}
In Section~3, by fitting them into bilateral Bailey lemmas, we  obtain several Rogers--Ramanujan type identities and also identities on false theta functions, such as
\begin{align*}
1+\sum_{n=1}^\infty\frac{(-q^3;q^3)_{n-1}q^{\binom{n+1}{2}}}{(-q)_{n-1}(q)_n(q;q^2)_n}
=\frac{1}{(q;q^2)_\infty(q^2,q^3,q^9,q^{10};q^{12})_\infty},
\end{align*}
and
\begin{align*}
\sum_{n=0}^\infty(-1)^n\frac{(q^3;q^3)_n}{(q)_{2n+1}}q^{\binom{n+1}{2}}
=\sum_{n=0}^\infty q^{3n(3n+1)}(1-q^{12n+6}).
\end{align*}
In Section~4, by using the general bilateral Bailey lattice due to Dousse, Jouhet and Konan, we derive an asymmetric bilateral Bailey lemma (as given in Lemma \ref{bbl-ls}),  which leads to identities on Appell--Lerch series.
Furthermore, the asymmetric bilateral Bailey pairs can also lead to some identities of Andrews--Gordon type, the details can be seen in Section~5. In Section~6, by applying  Andrews and Warnaar's bilateral Bailey lemmas, we obtain some identities on false theta functions and generalized Hecke--type series.

\section{Preliminaries}
Throughout this paper, we adopt  standard notations and terminologies
for $q$-series as given in \cite{GR04}. We assume that $|q|<1$, then the $q$-shifted factorials are defined by
\[
(a)_n=(a;q)_n=\begin{cases}
1, & \text{\it if $n=0$}, \\[2mm]
\prod_{j=1}^{n}(1-aq^{j-1}), & \text{\it if $n\geq 1$}, \\[2mm]
\end{cases}
\]
and
\[
(a)_\infty=(a;q)_\infty=\prod_{n=0}^\infty (1-aq^n).
\]
There are more compact notations for the multiple $q$-shifted factorials:
\begin{align*}
&(a_1,a_2,\cdots,a_m)_n=(a_1,a_2,\cdots,a_m;q)_n=(a_1;q)_n(a_2;q)_n \cdots(a_m;q)_n,\\
&(a_1,a_2,\cdots,a_m)_{\infty}=(a_1,a_2,\cdots,a_m;q)_{\infty}=(a_1;q)_{\infty}(a_2;q)_{\infty}\cdots
(a_m;q)_{\infty}.
\end{align*}
The $q$-binomial coefficients, or the Gaussian polynomials are given by
\[
{n \brack k} =
  \dfrac{(q;q)_n}
{(q;q)_k(q;q)_{n-k}}.
\]

The $q$-binomial theorem \cite[(3.3.6)]{And84}, also called Cauchy's binomial theorem, is known as
\begin{align}
(z)_n=\sum^{n}_{j=0}{n\brack j}(-1)^j q^{\binom{j}{2}} z^j .\label{q-binomial-thm}
\end{align}

The famous Jacobi's triple product identity \cite[(II.28)]{GR04} is
\begin{align}
(z, q/z, q; q)_\infty=\sum_{n=-\infty}^{\infty}(-1)^nq^{\binom{n}{2}}z^n. \label{JTP}
\end{align}

For $|ab|<1$, Ramanujan's theta function \cite[p.35]{AB12} is given by
\begin{align*}
f(a,b):=\sum_{n=-\infty}^{\infty} a^{\frac{n(n+1)}{2}} b^{\frac{n(n-1)}{2}}=(-a, -b, ab; ab)_\infty,
\end{align*}
which satisfies that \cite[p.46]{AB12}
\begin{align}
f(a,b)=f(a^3 b, ab^3)+af(b/a, a^5 b^3). \label{f(a,b)}
\end{align}

In 1996, Berkovich, McCoy and Schilling  gave the following bilateral Bailey lemma in \cite{BMS96}.

\begin{lemma}[BMS' bilateral Bailey lemma] \label{bbpinfinite}
If $(\alpha_n,\,\beta_n)$ is a bilateral Bailey pair related to $a$, then
\begin{align}
&\sum_{n=-\infty}^{\infty} (x, y)_{n}\left(\frac{aq}{xy}\right)^{n}\beta_{n}(a,q)\\
&\qquad
=\frac{(aq/x,aq/y)_{\infty}}{(aq,aq/xy)_{\infty}}\sum_{n=-\infty}^{\infty} \frac{(x,y)_n}{(aq/x,aq/y)_{n}} \left(\frac{aq}{x y}\right)^{n}\alpha_{n}(a,q).\nonumber
\end{align}
\end{lemma}

Then in \cite{Chu09}, Chu and Zhang generalized the above result to the following bilateral Bailey chain.

\begin{lemma}\label{bbl}
Let $\lambda$ be a nonnegative integer. For a bilateral sequence $\{\alpha_n(1,q)\}_{n\in\mathbb{Z}}$, define the unilateral sequence $\{\beta_n(1,q)\}_{n\in\mathbb{N}}$ by
\begin{align}
\beta_{n}=\sum_{k=-n-\lambda}^{n} \frac{\alpha_{k}}{(q)_{n-k}( q)_{n+k+\lambda}}, \label{asymmetricbb}
\end{align}
then
\begin{align}
\beta_{n}^{\prime}+\mathfrak{R}_{\lambda}(n) = \sum_{k=-n-\lambda}^{n} \frac{\alpha_{k}^{\prime}}{(q)_{n-k}(q)_{n+k+\lambda}},\label{asymmetricbb-newpair}
\end{align}
provided that
\begin{align*}		
\alpha_{n}^{\prime}& = \frac{(x,y)_{n}}{(q^{1+\lambda}/x,q^{1+\lambda}/y)_{n}} \left(\frac{q^{1+\lambda}}{xy}\right)^{n}\alpha_{n},\\
\beta_{n}^{\prime}& = \frac{1}{(q^{1+\lambda}/x,q^{1+\lambda}/y)_{n}} \sum_{k=0}^{n}\frac{(q^{1+\lambda}/xy)_{n-k}\ (x,y)_{k}}{(q)_{n-k}}\left(\frac{q^{1+\lambda}}{xy}\right)^{k}\beta_{k},
\end{align*}
and the error--term $\mathfrak{R}_{\lambda}(n)$ is given by
\begin{align*}
\mathfrak{R}_{\lambda}(n) &=\sum_{k=1}^{\lambda-1} \frac{\alpha_{-k}}{(q^{1+\lambda} / x, q^{1+\lambda} / y)_{n}}  \sum_{i=1}^{\min \{k,\lambda-k\}} \frac{q^{i^{2}-i \lambda}}{(q)_{k-i}(q)_{\lambda-k-i}} \frac{\left(q^{1+\lambda} / x y\right)_{n+i}}{(q)_{n+i}(q/x, q/y)_{i}}.
\end{align*}
Note that, when $\lambda=0, 1$, the error--term $\mathfrak{R}_{\lambda}(n)=0$.
\end{lemma}

In \eqref{asymmetricbb}, when $\lambda=0$, $(\alpha_n,\beta_n)$ is also called a {\it symmetric bilateral Bailey pair};  and if $\lambda\neq 0$, it forms an {\it asymmetric bilateral Bailey pair}. By substituting $ \{\alpha^{\prime}_{n}\}_{n\in\mathbb{Z}} $ and $\{\beta_{n}^{\prime}\}_{n\in\mathbb{N}} $ into \eqref{asymmetricbb-newpair} and letting $ n\rightarrow \infty $, the following result can be obtained.

\begin{lemma}\label{bblit}
For two sequences $\{\alpha_n(1,q)\}_{n\in\mathbb{Z}}$, $\{\beta_n(1,q)\}_{n\in\mathbb{N}}$ subject to \eqref{asymmetricbb},  there holds
\begin{align}
&\mathfrak{R}_{\lambda}+\sum_{k=0}^{\infty} \left(\frac{q^{1+\lambda}}{xy}\right)^{k}(x, y)_{k}\ \beta_{k} \label{bblemma}\\
&\ =\frac{(q^{1+\lambda}/x,q^{1+\lambda}/y)_{\infty}} {(q,q^{1+\lambda}/xy)_{\infty}}\sum_{k=-\infty}^{+\infty} \frac{(x,y)_k}{(q^{1+\lambda}/x,q^{1+\lambda}/y)_{k}} \left(\frac{q^{1+\lambda}}{xy}\right)^{k}\alpha_{k},\nonumber
\end{align}
where the error--term is given by
\begin{align*}
\mathfrak{R}_{\lambda}=\sum_{1\leq i \leq n < \lambda} \frac{q^{i^{2}-i\lambda}\ \alpha_{-n}} {(q)_{n-i}\ (q)_{\lambda-n-i}\ (q/x, q/y)_{i}}.
\end{align*}
\end{lemma}

Specially, when $\lambda=1$, we are led to the following weak forms of Lemma \ref{bblit}.

\begin{lemma}\label{nonsymbbailey} For asymmetric bilateral Bailey pair $(\alpha_n(1,q), \ \beta_n(1,q))$ with $\lambda=1$, we have
\begin{subequations}
\begin{align}
&\sum_{n=0}^{\infty}q^{n(n+1)}\beta_n
=\frac{1}{(q)_\infty}\sum_{n=-\infty}^{\infty}q^{n(n+1)}\alpha_n, \label{nonsymbbailey1}\\
&\sum_{n=0}^{\infty}q^{\binom{n+1}{2}}(-q)_n\beta_n
=\frac{(-q)_\infty}{(q)_\infty}
\sum_{n=-\infty}^{\infty}q^{\binom{n+1}{2}}\alpha_n, \label{nonsymbbailey2}\\
&\sum_{n=0}^{\infty}(-1)^nq^{\binom{n+1}{2}}(q)_n\beta_n
=\sum_{n=0}^{\infty}(-1)^nq^{\binom{n+1}{2}}(\alpha_n+\alpha_{-n-1}). \label{nonsymbbailey4}
\end{align}
\end{subequations}
\end{lemma}

\begin{proof}
In Lemma \ref{bblit}, setting $a=1,\lambda=1$, and taking $x\rightarrow\infty, y\rightarrow\infty$; $x\rightarrow\infty, y\mapsto -q$, it leads to \eqref{nonsymbbailey1} and \eqref{nonsymbbailey2}, respectively. To derive \eqref{nonsymbbailey4},  we note that
\[
\lim_{y\rightarrow q}\frac{(y)_n}{(q^2/y)_n}=\begin{cases}
1, & \text{\it $n\geq 0$}, \\[2mm]
-1, & \text{\it $n<0$},
\end{cases}
\]
which was also stated in \cite{Chu10}.
Thus by taking $x\rightarrow\infty, y\mapsto q$, the right hand side of \eqref{bblemma} becomes  to
\begin{align*}
\sum_{n=0}^{\infty}(-1)^nq^{\binom{n+1}{2}}\alpha_n
-\sum_{n=-\infty}^{-1} (-1)^n q^{\binom{n+1}{2}} \alpha_n.
\end{align*}
Then substituting $n\mapsto -n-1$ into the second sum, it leads to the right hand side of \eqref{nonsymbbailey4}.
\end{proof}

Moreover, from an asymmetric bilateral Bailey pair $(\alpha_{n}(1,q), \beta_{n}(1,q))$ with $\lambda=1$, by employing the bilateral Bailey chain as given in Lemma \ref{bbl} iteratively, we can derive a sequence of bilateral Bailey pairs, which leads to the following identities on multisums.

\begin{corollary}\label{weakiterat}
If $(\alpha_{n}(1,q), \ \beta_{n}(1,q))$ is an asymmetric bilateral Bailey pair with $\lambda=1$, then
\begin{subequations}
\begin{align}
&\sum_{n_1\geq\cdots\geq n_k\geq 0}\frac{q^{\sum_{i=1}^{k}n_i(n_i+1)}}{(q)_{n_1-n_2}\cdots (q)_{n_{k-1}-n_k}}\beta_{n_k}
=\frac{1}{(q)_\infty}\sum_{n=-\infty}^{\infty}q^{kn(n+1)}\alpha_n, \label{weakiterate1}\\
&\sum_{n_1\geq\cdots\geq n_k\geq 0}\frac{q^{\sum_{i=1}^{k}\binom{n_i+1}{2}}(-q)_{n_k}}
{(q)_{n_1-n_2}\cdots (q)_{n_{k-1}-n_k}}\beta_{n_k}
=\frac{(-q)_\infty}{(q)_\infty}\sum_{n=-\infty}^{\infty}
q^{k\binom{n+1}{2}}\alpha_n, \label{weakiterate2}\\
&\sum_{n_1\geq\cdots\geq n_k\geq 0} \frac{(-1)^{\sum_{i=1}^{k}n_i} q^{\sum_{i=1}^{k} \binom{n_i+1}{2}} \ (q)_{n_k}} {(q)_{n_1-n_2}\cdots (q)_{n_{k-1}-n_k}}\beta_{n_k}\label{weakiterate4}\\
&\qquad\qquad\qquad\qquad\qquad
=\sum_{n=0}^{\infty} (-1)^{kn} q^{k\binom{n+1}{2}} (\alpha_n+\alpha_{-n-1}). \nonumber
\end{align}
\end{subequations}
\end{corollary}

\begin{proof}
Letting $a=1$, $\lambda=1$, $x\rightarrow\infty, y\rightarrow\infty$ in the bilateral Bailey chain as given in Lemma \ref{bbl}, we  get
\begin{align*}
&\alpha^{(1)}_n=q^{n(n+1)} \alpha_n,\quad \
\beta^{(1)}_n=\sum_{j=0}^{n} \frac{q^{j(j+1)}}{(q)_{n-j}}\beta_{j}.
\end{align*}
Iterating the above relation $k-1$ times, we obtain the following bilateral Bailey pair
\begin{align*}
&\alpha^{(k-1)}_n=q^{(k-1)n(n+1)} \alpha_n,\\
&\beta^{(k-1)}_n=\sum_{n_{k-1}=0}^{n}\sum_{n_{k-2}=0}^{n_{k-1}}\sum_{n_{k-3}=0}^{n_{k-2}} \cdots \sum_{n_{1}=0}^{n_{2}} \frac{q^{\sum_{i=1}^{k-1}n_i(n_i+1)}}{(q)_{n-n_{k-1}(q)_{n_{k-1}-n_{k-2}}\cdots (q)_{n_2-n_1}}}\beta_{n_1}.
\end{align*}
Substituting the above bilateral Bailey pair into \eqref{nonsymbbailey1}, we have
\begin{align*}
&\sum_{n\geq 0} \sum_{n_{k-1}=0}^{n}\sum_{n_{k-2}=0}^{n_{k-1}}\cdots \sum_{n_{1}=0}^{n_{2}} \frac{q^{n(n+1)+\sum_{i=1}^{k-1}n_i(n_i+1)}}{(q)_{n-n_{k-1}(q)_{n_{k-1}-n_{k-2}}\cdots (q)_{n_2-n_1}}}\beta_{n_1} \\
&\qquad\qquad\qquad\qquad\qquad\qquad\qquad\quad= \frac{1}{(q)_\infty}\sum_{n=-\infty}^\infty q^{kn(n+1)}\alpha_n.
\end{align*}
Then by replacing the variables on the left hand side with $n\mapsto n_1$ and $n_j\mapsto n_{k+1-j}$ for $1\leq j \leq k-1$, we arrive at \eqref{weakiterate1}. Simliarly, \eqref{weakiterate2} and \eqref{weakiterate4} can be proved by applying the weak forms  \eqref{nonsymbbailey2} and \eqref{nonsymbbailey4}, respectively.
\end{proof}

Recently, Dousse, Jouhet and Konan \cite{JFI25} gave a general bilateral Bailey lattice, which is stated as follows.

\begin{lemma}\label{lem:baileylattices}
If $(\alpha_n,\,\beta_n)$ is a bilateral Bailey pair related to $a$, then  $(\alpha'_n,\,\beta'_n)$ is a bilateral Bailey pair related to $a/q$, where
\begin{align}\label{eq:baileylatticesalpha}
\alpha'_n=\frac{1-a}{1-b}\left(\frac{1-bq^n}{1-aq^{2n}} \alpha_n -\frac{q^{n-1}(aq^{n-1}-b)}{1-aq^{2n-2}} \alpha_{n-1} \right),
\end{align}
and
\begin{align}\label{eq:baileylatticesbeta}
\beta'_n=\frac{(bq)_n}{(b)_n}\beta_n.
\end{align}
\end{lemma}

Using the bilateral Bailey lattice described above, we derive the following asymmetric bilateral Bailey lemma.

\begin{lemma}\label{bbl-ls}
For asymmetric bilateral Bailey $(\alpha_n(1,q),\ \beta_n(1,q))$ with $\lambda=1$, we have
\begin{align}\label{asymmetric-bilateral-Bailey-lemma}
&\sum_{n=0}^{\infty} (x, y)_{n} (1-bq^n) \left(\frac{q}{xy}\right)^{n}\ \beta_{n}\\
&=\frac{(q/x,q/y)_{\infty}} {(q,q/xy)_{\infty}} \sum_{n=-\infty}^{\infty} \frac{(x,y)_n \ (q/xy)^n}{(q/x,q/y)_{n}} \Big(\frac{1-bq^n} {1-q^{2n+1}}\ \alpha_{n} - \frac{q^{n-1} (q^n-b)}{1-q^{2n-1}}\ \alpha_{n-1} \Big).\nonumber
\end{align}
\end{lemma}

\begin{proof}
Note that an asymmetric bilateral Bailey pair $(\alpha_{n}(1,q), \beta_n(1,q))$ with $a=1$ and $ \lambda=1$ can be rewritten as
\begin{align*}
(1-q)\beta_{n}(1,q)=\sum_{k=-n-1}^{n} \frac{\alpha_{k}(1,q)}{(q)_{n-k} (q^2)_{n+k}},
\end{align*}
which leads to the following bilateral Bailey pair with $a=q$
\begin{align*}
\big(\alpha_{n}(1,q),   (1-q)\beta_{n}(1,q) \big).
\end{align*}
Then inserting the above bilateral Bailey pair  into the bilateral Bailey lattice as given in Lemma \ref{lem:baileylattices} with $a=q$, and dividing $1-q$ on both hand sides of the resulting identity, we obtain a new bilateral Bailey pair with $a=1$
\begin{align*}
&\alpha_{n}'(1,q)=\frac{1}{1-b} \Big(\frac{1-bq^n}{1-q^{2n+1}}\alpha_{n}(1,q) - \frac{q^{n-1}(q^n-b)}{1-q^{2n-1}}\alpha_{n-1}(1,q) \Big),\\[2mm]
&\beta_{n}'(1,q)=\frac{(bq)_n}{(b)_n}\beta_{n}(1,q).
\end{align*}
Substituting $(\alpha_{n}'(1,q), \beta_{n}'(1,q))$ into the bilateral Bailey Lemma \ref{bbpinfinite}, the proof is complete.
\end{proof}

\section{Two asymmetric bilateral Bailey pairs}

In this section, we give two asymmetric bilateral Bailey pairs following Rogers' way to prove Rogers--Ramanujan identities in \cite{Rog17}, in which Rogers couched his work in
terms of Fourier series.  As pointed by Andrews \cite{And19}, the following formula was the  finite version of the result given by Rogers \cite{Rog17}
\begin{align}
(-qe^{i\theta}, -qe^{-i\theta})_n =\prod_{j=1}^n (1+2 q^j \cos\theta + q^{2j}). \label{triangle fomular}
\end{align}

Firstly, we expand the above formula by applying the $q$-binomial theorem to derive the following asymmetric bilateral Bailey pair.

\begin{lemma}
We have asymmetric bilateral Bailey with $a=1$ and $\lambda=1$
\begin{align}
(\alpha_n(1,q),\ \beta_n(1,q))=\Big(z^{n}q^{n+1\choose 2},\ \frac{(-1/z)_{n+1}(-zq)_n}{(q)_{2n+1}}\Big).\label{asymmetricbb1}
\end{align}
\end{lemma}

\begin{proof}
We first expand  the left hand side of \eqref{triangle fomular}  to be
\begin{align*}
(-q/z, -qz)_n & =(-qz)_n\ z^{-n} q^{\binom{n+1}{2}} (-zq^{-n})_n\\
&=\frac{q^{\binom{n+1}{2}}}{(1+z) z^n} (-zq^{-n})_{2n+1}.
\end{align*}
Then by applying the $q$-binomial theorem \eqref{q-binomial-thm}, it leads to that
\begin{align*}
(-q/z, -qz)_n &=\frac{q^{\binom{n+1}{2}}}{(1+z)z^n}\sum_{j=0}^{2n+1}{2n+1 \brack j}(-1)^jq^{\binom{j}{2}}(-zq^{-n})^j.
\end{align*}
Substituting $j\mapsto n-j$ into the sum on the right hand side and then setting $z\mapsto z^{-1}$, it turns to be
\begin{align}\label{Andrewidentity1-transform}
(-1/z)_{n+1}(-zq)_n =\sum_{j=-n-1}^{n} {2n+1 \brack n-j} q^{\binom{j+1}{2}}z^{j},
\end{align}
which completes the proof by comparing with the definition of the asymmetric bilateral Bailey pair \eqref{asymmetricbb} with $a=1$ and  $\lambda=1$.
\end{proof}

Specially, when  $n\rightarrow\infty$ in \eqref{Andrewidentity1-transform}, it implies to Jacobi's triple product identity \eqref{JTP}, which means that \eqref{Andrewidentity1-transform} can be seen as a truncated form of Jacobi's triple product identity. If we set $z=e^{i\theta}$ in \eqref{Andrewidentity1-transform}, it leads to Andrews' identity \eqref{Andrewidentity1}. Taking $q\mapsto q^2, z\mapsto zq^{-1}$ in \eqref{asymmetricbb1}, it yields the asymmetric bilateral Bailey pair as given in \cite[P.187]{AW07}. Furthermore, we find that the bilateral Bailey pair \eqref{asymmetricbb1} is also a special case of the one given in \cite[P. 366]{Chu09} with $\delta=1$.

By inserting the asymmetric bilateral Bailey pair \eqref{asymmetricbb1} into bilateral Bailey lemmas, it will lead to identities of Rogers--Ramanujan type.  By fitting the asymmetric bilateral Bailey pair \eqref{asymmetricbb1} into \eqref{nonsymbbailey1} and \eqref{nonsymbbailey2},  we obtain the following results, respectively. The infinite products on the right hand sides are obtained by using Jacobi's triple product identity \eqref{JTP}.

\begin{theorem}\label{nonsymbbailey30}
We have
\begin{align}
&\sum_{n=0}^\infty \frac{(-q/z,-zq)_n}{(q)_{2n+1}} q^{n(n+1)} =\frac{(-q^3/z,-zq^3,q^3;q^3)_\infty}{(q)_\infty},\label{nonsymbbailey301}\\
&\sum_{n=0}^\infty \frac{(-q/z,-zq,-q)_n}{(q)_{2n+1}} q^{\binom{n+1}{2}}
=\frac{(-q^2/z,-zq^2;q^2)_\infty}{(q;q^2)^2_\infty}.\label{nonsymbbailey302}
\end{align}
\end{theorem}

Note that if we take $z=e^{i\theta}$ and then set $\theta$ to be some special values in Theorem \ref{nonsymbbailey30}, it will reduce to some identities as given in  \cite{And19, Yao25}. For example, when $z=e^{i\theta}$ and $\theta=0$ in \eqref{nonsymbbailey301}, we have
\begin{align*}
\sum_{n=0}^\infty\frac{(-q)^2_n}{(q)_{2n+1}}q^{n(n+1)}
=\frac{(-q^3,-q^3,q^3;q^3)_\infty}{(q)_\infty},
\end{align*}
which can be seen in \cite[Theorem 4.3]{And19}.

Then by inserting the asymmetric bilateral Bailey pair \eqref{asymmetricbb1} into \eqref{nonsymbbailey4}, we get the following identity on false theta functions.

\begin{theorem}\label{nonsymbbailey31}
We have
\begin{align}
\sum_{n=0}^\infty(-1)^n\frac{(-1/z)_{n+1} (-zq, q)_n}{(q)_{2n+1}}q^{\binom{n+1}{2}}
=\Big(\sum_{n=0}^\infty - \sum_{n=-\infty}^{-1} \Big) \ (-1)^n q^{n(n+1)} z^{n}.
\end{align}
\end{theorem}

The notion of false theta functions was introduced by Rogers \cite{Rog17}, which can be seen as instances of classical theta series  except for an alteration of the signs of some of its terms.
In 2003, Berndt and Yee \cite{BY03} studied the combinatorial proofs for some identities on false theta functions. In Chapter 9 of Ramanujan's lost notebook \cite{AB05}, Andrews and Berndt established some false theta function identities arising from the Rogers--Fine identity.
More related  works on false theta functions can be seen in \cite{BMS96,Chu10}.

By letting $z=e^{i\theta}$ in Theorem \ref{nonsymbbailey31}, and taking $\theta=0, \frac{\pi}{2}, \frac{\pi}{3}, \frac{2\pi}{3}$, respectively, we obtain the following results.

\begin{corollary}
We have
\begin{align}
&\sum_{n=0}^\infty(-1)^n\frac{(-q)_n}{(q;q^2)_{n+1}}q^{\binom{n+1}{2}}
=\sum_{n=0}^\infty(-1)^nq^{n(n+1)},\label{nonsymbbailey310}\\
&\sum_{n=0}^\infty(-1)^n\frac{(-q^2;q^2)_n}{(q^{n+1})_{n+1}}q^{\binom{n+1}{2}}
=\sum_{n=0}^\infty(-1)^nq^{2n(2n+1)}(1+q^{4n+2}),\label{nonsymbbailey311}\\
&\sum_{n=0}^\infty(-1)^n\frac{(q^3;q^3)_n}{(q)_{2n+1}}q^{\binom{n+1}{2}}
=\sum_{n=0}^\infty q^{3n(3n+1)}(1-q^{12n+6}),\label{nonsymbbailey312}\\
&\sum_{n=0}^\infty(-1)^n\frac{(-q^3;q^3)_n}{(q^{n+1})_{n+1}(-q)_n}
q^{\binom{n+1}{2}}\label{nonsymbbailey313}\\
&\qquad\qquad\qquad
=\sum_{n=0}^\infty(-1)^nq^{3n(3n+1)}(1+2q^{6n+2}+q^{12n+6}).\nonumber
\end{align}
\end{corollary}

\begin{proof}
We mainly give the proof of identity \eqref{nonsymbbailey313}, and others are similar. Taking $\theta=\frac{2\pi}{3}$  into Theorem \ref{nonsymbbailey31}, and multiplying both sides by $\frac{1}{1+e^{-\frac{2}{3}\pi i}}$, for the left hand side we have
\begin{align*}
&\sum_{n=0}^\infty\frac{(-1)^nq^{\binom{n+1}{2}}}{(q^{n+1})_{n+1}}
(-qe^{-\frac{2}{3}\pi i},-qe^{\frac{2}{3}\pi i})_n\\
&=\sum_{n=0}^\infty\frac{(-1)^nq^{\binom{n+1}{2}}}{(q^{n+1})_{n+1}}
\prod_{j=1}^{n}(1+q^{j} e^{-\frac{2}{3}\pi i}) (1+q^{j} e^{\frac{2}{3}\pi i})\\
&=\sum_{n=0}^\infty\frac{(-1)^nq^{\binom{n+1}{2}}}{(q^{n+1})_{n+1}}
\prod_{j=1}^{n}(1-q^{j}+q^{2j})\\
&=\sum_{n=0}^\infty(-1)^n\frac{(-q^3;q^3)_n}{{(q^{n+1})_{n+1}}(-q)_n}
q^{\binom{n+1}{2}}.
\end{align*}
For the right hand side of Theorem \ref{nonsymbbailey31}, it becomes
\begin{align*}
\frac{1}{1+e^{-\frac{2}{3}\pi i}}\sum_{n=0}^\infty(-1)^nq^{n(n+1)}
(e^{\frac{2}{3}\pi in}+e^{-\frac{2}{3}\pi i(n+1)}).
\end{align*}
By considering the residue classes of $n$ modulo $3$ in the above expression, it turns to  be
\begin{align*}
\sum_{n=0}^\infty(-1)^nq^{3n(3n+1)}+2\sum_{n=0}^\infty
(-1)^nq^{(3n+1)(3n+2)}+\sum_{n=0}^\infty(-1)^nq^{(3n+2)(3n+3)},
\end{align*}
which completes the proof by simplifying.
\end{proof}

Next, from identity \eqref{Andrewidentity2} due to Andrews, we obtain another asymmetric bilateral Bailey pair.

\begin{lemma}
If $|z|< 1$, the following pair of function
\begin{align}
(\alpha_n(1,q),\ \beta_n(1,q))=
\Big(\frac{z^{-n} q^{\binom{n}{2}} (1-q^{2n+1})}{1-z},\ \frac{(-z,-1/z)_n}{(q)_{2n}}\Big) \label{asymmetricbb2}
\end{align}
forms an asymmetric bilateral Bailey pair with $a=1$ and  $\lambda=1$.
\end{lemma}

\begin{proof}
Denote $z=e^{i\theta}$, then $x=\cos \theta =\frac{z +z^{-1}}{2}$, and $\sin \theta =\frac{z-z^{-1}}{2i}$. By noting the classical formula \eqref{triangle fomular},
the left hand side of identity \eqref{Andrewidentity2} turns to be
\begin{align*}
( -z, -1/z)_n.
\end{align*}
Then recall the definition of Chebyshev polynomials of the  fourth kind
\begin{align*}
W_n(x)=\frac{\sin(n+\frac{1}{2})\theta}{\sin\frac{1}{2}\theta}=\frac{e^{i(n+1)\theta}-e^{-in\theta}}{e^{i\theta}-1},
\end{align*}
hence identity \eqref{Andrewidentity2} becomes
\begin{align*}
(-z, -1/z)_n =\sum_{j=0}^n \frac{q^{j\choose 2} (1-q^{2j+1}) (q)_{2n}}{(q)_{n-j} (q)_{n+j+1}}
\frac{(z^{j+1}-z^{-j})}{z-1}.
\end{align*}
Taking $j\mapsto -j-1$ into the first sum with $z^{j+1}$ of the right hand side, it can be rewritten as
\begin{align}\label{Andrewidentity2-transform}
\frac{(-z, -1/z)_n}{(q)_{2n}}
=\frac{1}{1-z} \sum_{j=-n-1}^n \frac{z^{-j} q^{j\choose 2} (1-q^{2j+1})}{(q)_{n-j} (q)_{n+j+1}},
\end{align}
which completes the proof.
\end{proof}

Identity \eqref{Andrewidentity2-transform} also can be seen as a truncated form of the Jacobi's triple product identity.

Now, by inserting the asymmetric bilateral Bailey pair \eqref{asymmetricbb2} into formulas \eqref{nonsymbbailey1}--\eqref{nonsymbbailey4}, we obtain the following identities.

\begin{theorem}\label{asymbbaileyweak2}
We have
\begin{align}
&\sum_{n=0}^\infty\frac{(-1/z, -z)_n}{(q)_{2n}} q^{n(n+1)}\label{asymbbaileyweak20}\\
&\ =\frac{1}{(1-z) (q)_\infty} \big((-q^2/z, -zq,q^3;q^3)_\infty -z(-q/z, -zq^2,q^3;q^3)_\infty \big),\nonumber\\
&\sum_{n=0}^\infty \frac{(-1/z, -z)_n}{(q)_n(q;q^2)_n} q^{\binom{n+1}{2}}
=\frac{(-q)_\infty}{(q)_\infty} (-q/z, -zq,q^2;q^2)_\infty,\label{asymbbaileyweak21}\\
&\sum_{n=0}^\infty \frac{(-1)^n (-1/z, -z)_n}
{(-q)_n (q;q^2)_n} q^{\binom{n+1}{2}} \label{asymbbaileyweak41}\\
&\qquad\qquad\qquad
=\frac{1}{1-z} \Big(\sum_{n=0}^{\infty} -  \sum_{n=-\infty}^{-1} \Big) (-1)^{n+1} q^{n^2} (1-q^{2n+1}) z^{n+1}.\nonumber
\end{align}
\end{theorem}

Letting $z=e^{i\theta}$ and taking some special values for $\theta$ in Theorem \ref{asymbbaileyweak2}, it leads to a series of identities. By taking $\theta=\frac{\pi}{2}, \frac{\pi}{3}, \frac{2\pi}{3}$ in \eqref{asymbbaileyweak20}, we obtain the following identities on Ramanujan's theta functions, respectively.

\begin{corollary}
We have
\begin{align}
&\sum_{n=0}^\infty \frac{(-1;q^2)_n}{(q)_{2n}} q^{n(n+1)}=\frac{1}{(q)_\infty} \big(f(-q^5, -q^7) -q f(-q, -q^{11}) \big), \label{asymbbaileyweak20-coro1}\\
&1+3\sum_{n=1}^\infty \frac{(q^3;q^3)_{n-1}}{(q)_{n-1} (q)_{2n}} q^{n(n+1)} \label{asymbbaileyweak20-coro2}\\
&\ =\frac{1}{(q)_\infty} \big(f(-q^{12}, -q^{15}) -q f(-q^6, -q^{21}) +2q^2 f(-q^{3}, -q^{24}) \big),\nonumber\\
&1+\sum_{n=1}^\infty \frac{(-q^3;q^3)_{n-1}}{(-q)_{n-1}(q)_{2n}} q^{n(n+1)} =\frac{1}{(q)_\infty} \big( f(q^{12}, q^{15}) -qf(q^{6}, q^{21}) \big). \label{asymbbaileyweak20-coro3}
\end{align}
\end{corollary}

\begin{proof}
In order to get \eqref{asymbbaileyweak20-coro1} more convenience,  we simplify the identity by inserting the asymmetric bilateral Bailey pair \eqref{asymmetricbb2} into formula \eqref{nonsymbbailey1}, then setting $z=e^{i\theta}, \theta=\frac{\pi}{2}$, we can get the left hand immediately. For the right hand, it becomes
\begin{align*}
\frac{1}{(1-i)(q)_\infty} \sum_{n=-\infty}^{\infty} q^{n(n+1)+\binom{n}{2}} (1-q^{2n+1}) i^{-n},
\end{align*}
considering the residues of $n$ modulo 4, it can be written as
\begin{align*}
&\frac{1}{(1-i)(q)_\infty}
\big( \sum_{n=-\infty}^{\infty}
q^{4n(4n+1)+\binom{4n}{2}} (1-q^{8n+1})
-i \sum_{n=-\infty}^{\infty}
q^{(4n+1)(4n+2)+\binom{4n+1}{2}} \\
&\qquad\qquad\qquad
(1-q^{8n+3}) -\sum_{n=-\infty}^{\infty}
q^{(4n+2)(4n+3)+\binom{4n+2}{2}} (1-q^{8n+5}) \\
&\qquad\qquad\qquad
+i \sum_{n=-\infty}^{\infty}
q^{(4n+3)(4n+4)+\binom{4n+3}{2}} (1-q^{8n+7}) \big),
\end{align*}
we can see that
\begin{align*}
&\sum_{n=-\infty}^{\infty}
q^{4n(4n+1)+\binom{4n}{2}} (1-q^{8n+1})
+i \sum_{n=-\infty}^{\infty}
q^{(4n+3)(4n+4)+\binom{4n+3}{2}} (1-q^{8n+7})\\ &\quad
=(1-i)\sum_{n=-\infty}^{\infty}
q^{4n(4n+1)+\binom{4n}{2}} (1-q^{8n+1}),
\end{align*}
by taking $n \mapsto -n-1$ for the second term. Similarly
\begin{align*}
&-i \sum_{n=-\infty}^{\infty}
q^{(4n+1)(4n+2)+\binom{4n+1}{2}}(1-q^{8n+3}) -\sum_{n=-\infty}^{\infty}
q^{(4n+2)(4n+3)+\binom{4n+2}{2}} (1-q^{8n+5})\\ &\quad
=(1-i)\sum_{n=-\infty}^{\infty}
q^{(4n+1)(4n+2)+\binom{4n+1}{2}}(1-q^{8n+3}),
\end{align*}
hence combing with the Jacobi's triple product identity \eqref{JTP}, it reduces to
\begin{align*}
&\frac{1}{(q)_\infty} \big((-q^{22}, -q^{26}, q^{48}; q^{48})_\infty - q (-q^{14}, -q^{34}, q^{48}; q^{48})_\infty\\
&\qquad\qquad + q^2 (-q^{10}, -q^{38}, q^{48}; q^{48})_\infty - q^5 (-q^{2}, -q^{46}, q^{48}; q^{48})_\infty \big),
\end{align*}
by noting \eqref{f(a,b)}, we can see that
\begin{align*}
&(-q^{22}, -q^{26}, q^{48}; q^{48})_\infty- q^5 (-q^{2}, -q^{46}, q^{48}; q^{48})_\infty\\
&\ =f(q^{22}, q^{26}) -q^5 f(q^{2}, q^{46})\\   &\ =f(-q^5, -q^7),
\end{align*}
and
\begin{align*}
&(-q^{14}, -q^{34}, q^{48}; q^{48})_\infty- q (-q^{10}, -q^{38}, q^{48}; q^{48})_\infty =f(-q, -q^{11}),
\end{align*}
identity \eqref{asymbbaileyweak20-coro1} is proved.

Now let $\theta=\frac{\pi}{3}$, take parity of n modulo 6 for the right hand, it can reduce to
\begin{align*}
&\frac{1}{(q)_\infty} \big((-q^{51}, -q^{57}, q^{108}; q^{108})_\infty -q (-q^{39}, -q^{69}, q^{108}; q^{108})_\infty\\
&\qquad +2 q^2 (-q^{33}, -q^{75}, q^{108}; q^{108})_\infty -2 q^5 (-q^{21}, -q^{87}, q^{108}; q^{108})_\infty\\
&\qquad +q^7 (-q^{15}, -q^{93}, q^{108}; q^{108})_\infty -q^{12} (-q^{3}, -q^{105}, q^{108}; q^{108})_\infty \big),
\end{align*}
simplify it by formula \eqref{f(a,b)}, we can get \eqref{asymbbaileyweak20-coro2}.

For \eqref{asymbbaileyweak20-coro3}, we can get it by letting  $\theta=\frac{2\pi}{3}$, and taking parity of n modulo 3 for the right hand.
\end{proof}

Then by taking $\theta=0, \frac{\pi}{2}, \frac{\pi}{3}, \frac{2\pi}{3}$ in \eqref{asymbbaileyweak21}, respectively, we get  the following identities of Rogers--Ramanujan type.

\begin{corollary}
We have
\begin{align*}
&\sum_{n=0}^\infty \frac{(-1)^2_n \ q^{\binom{n+1}{2}}}{(q)_n(q;q^2)_n}
=(-q)^2_\infty (-q;q^2)^2_\infty,\\
&\sum_{n=0}^\infty \frac{(-1;q^2)_n \ q^{\binom{n+1}{2}}}{(q)_n(q;q^2)_n}
=(-q)^2_\infty (-q^2;q^4)_\infty,\\
&1+3\sum_{n=1}^\infty \frac{(q^3;q^3)_{n-1} \ q^{\binom{n+1}{2}}}{(q)_{n-1} (q)_n(q;q^2)_n}
=(-q)^3_\infty(q^3;q^6)_\infty,\\
&1+\sum_{n=1}^\infty \frac{(-q^3;q^3)_{n-1} \ q^{\binom{n+1}{2}}}{(-q)_{n-1}(q)_n(q;q^2)_n}
=\frac{1}{(q;q^2)_\infty (q^2,q^3,q^9,q^{10};q^{12})_\infty}.
\end{align*}
\end{corollary}

 Setting  $\theta=\frac{\pi}{2},\frac{\pi}{3},\frac{2\pi}{3}$  in \eqref{asymbbaileyweak41}, it leads to the   following identities on  false theta functions, respectively.
\begin{corollary}
We have
\begin{subequations}
\begin{align*}
&\sum_{n=0}^\infty \frac{(-1)^n(-1;q^2)_n}{(-q)_n(q;q^2)_n} q^{\binom{n+1}{2}}
=\sum_{n=0}^\infty (-1)^n q^{4n^2} (1-2q^{4n+1}+q^{8n+4}),\\
&1+3\sum_{n=1}^\infty\frac{(-1)^n(q^3;q^3)_{n-1}} {(q)_{n-1}(-q)_n(q;q^2)_n} q^{\binom{n+1}{2}}\\
&\qquad\qquad\qquad
=\sum_{n=0}^\infty q^{9n^2} (1-3q^{6n+1}+3q^{12n+4}-q^{18n+9}),\\
&1+\sum_{n=1}^\infty\frac{(-1)^n(-q^3;q^3)_{n-1}} {(-q)_{n-1}(-q)_n(q;q^2)_n} q^{\binom{n+1}{2}}\\
&\qquad\qquad\qquad
=\sum_{n=0}^\infty (-1)^n q^{9n^2} (1-q^{6n+1}-q^{12n+4}+q^{18n+9}).
\end{align*}
\end{subequations}
\end{corollary}

\section{Identities on Appell--Lerch series}

In this section, based on the asymmetric bilateral Bailey lemma as given in Lemma \ref{bbl-ls} and combining with our  asymmetric bilateral Bailey pairs \eqref{asymmetricbb1} and \eqref{asymmetricbb2}, we derive identities on Appell--Lerch series.

Recall that the Appell--Lerch series introduced by Appell \cite{App846} and Lerch \cite{Ler92} has the following form
\begin{equation}\label{appler}
\sum_{n=-\infty}^\infty \frac{(-1)^{\ell n}q^{\ell n(n+1)/2}b^n}{1-aq^n}.
\end{equation}

Firstly, by inserting the asymmetric bilateral Bailey pair \eqref{asymmetricbb1} into the asymmetric bilateral Bailey lemma \ref{asymmetric-bilateral-Bailey-lemma}, we get the following identity on Appell--Lerch series.

\begin{corollary}
We have
\begin{align}
&\sum_{n=0}^{\infty} (x, y)_{n} (1-bq^n) \left(\frac{q}{xy}\right)^{n} \frac{(-1/z)_{n+1} (-zq)_n}{(q)_{2n+1}}\\
&\ =\frac{(q/x,q/y)_{\infty}}{(q,q/xy)_{\infty}} \sum_{n=-\infty}^{\infty} \frac{(x,y)_n q^{\binom{n+1}{2}}(qz/xy)^n}{(q/x,q/y)_{n}}  \bigg(\frac{1-bq^n}{1-q^{2n+1}} - \frac{q^n-b}{zq (1-q^{2n-1})}\bigg).\nonumber
\end{align}
\end{corollary}

Consequently, letting $b=0,\ x,y\rightarrow\infty$, and then taking $z=1$ and $z=q$, respectively, it leads to the following identities
\begin{align*}
&\sum_{n=0}^{\infty}\frac{(-q)_n^2}{(q)_{2n+1}}q^{n^2}
=\frac{(-q, -q^2, q^3; q^3)_\infty}{(q)_\infty},\\
&\sum_{n=0}^{\infty}\frac{(-1)_n (-q)_{n+1}}{(q)_{2n+1}} q^{n^2}
=\frac{1}{(q)_\infty}((-q,-q^2,q^3;q^3)_\infty + q(-1,-q^3,q^3;q^3)_\infty),
\end{align*}
where the first one can be seen in    \cite[Entry 4.2.9]{AB09}.

Moreover, by taking $b=0, x\rightarrow\infty, y\mapsto -\sqrt{q}$, and $z=1$, then letting $q\mapsto q^2$, it  reduces to the following identity on Appell--Lerch series
\begin{align*}
\sum_{n=0}^{\infty}\frac{(-q^2;q^2)_n}{(1+q^{2n+1}) (q)_{2n+1}} q^{n^2}
=\frac{(-q;q^2)_\infty}{(q^2;q^2)_\infty} \sum_{n=-\infty}^{\infty} \frac{q^{n(2n+1)}}{1-q^{4n+2}}.
\end{align*}

Next, by fitting the asymmetric bilateral Bailey pair \eqref{asymmetricbb2} into formula \ref{asymmetric-bilateral-Bailey-lemma}, we obtain the following result.

\begin{corollary}\label{bi-Bailey-lattice2}
We have
\begin{align*}
&\sum_{n=0}^{\infty} (x, y)_{n} (1-bq^n) \left(\frac{q}{xy}\right)^{n} \frac{(-1/z, -z)_{n}} {(q)_{2n}}\\
&\quad =\frac{1}{1-z} \frac{(q/x,q/y)_{\infty}}{(q,q/xy)_{\infty}}
\sum_{n=-\infty}^{\infty} \frac{(x,y)_n q^{\binom{n}{2}}(q/xyz)^n}{(q/x,q/y)_{n}}
\big( (1-bq^n)-z(q^n-b) \big).
\end{align*}
\end{corollary}

Considering the special case when $b=1$ and $x,y\rightarrow\infty$ in  the above corollary, it implies the following result.

\begin{corollary}
We have
\begin{align*}
&\sum_{n=0}^{\infty} \frac{(-1/z, -z)_n}{(q)_{2n}} q^{n^2}\\
&\qquad
= \frac{1}{(1-z)(q)_\infty} \big((-q/z, -zq^2, q^3; q^3)_\infty - z(-q^2/z, -zq, q^3; q^3)_\infty\big).
\end{align*}
\end{corollary}

\begin{proof}
By letting $b=1$ and $x,y\rightarrow\infty$ in Corollary \ref{bi-Bailey-lattice2}, combing the Jacobi's triple product identity \eqref{JTP}, it leads to the following identity
\begin{align*}
&\sum_{n=0}^{\infty} \frac{(-1/z, -z)_n}{(q)_{2n}} (1-q^n) \ q^{n^2}\\
&\quad = \frac{1+z}{(1-z)(q)_\infty} \big((-q/z, -zq^2, q^3; q^3)_\infty - (-q^2/z, -zq, q^3; q^3)_\infty\big).
\end{align*}
Expanding the left hand side of the above identity and substituting it by  \eqref{asymbbaileyweak20}, we complete the proof. \end{proof}

\section{The identities of Andrews--Gordon type}

In this section, we obtain some Andrews--Gordon type identities by inserting the asymmetric bilateral Bailey pair \eqref{asymmetricbb1} into the multisums stated in Corollary \ref{weakiterat}.

Firstly, by taking the asymmetric bilateral Bailey pair \eqref{asymmetricbb1} into \eqref{weakiterate1}, we  get the following identity.

\begin{theorem}\label{miterateAndrews}
For $k\in \mathbb{N^+}$, we have
\begin{align}
&\sum_{n_1\geq\cdots\geq n_k\geq 0} \frac{(-q/z,-zq)_{n_k} \ q^{\sum_{i=1}^{k}n_i(n_i+1)}}{(q)_{n_1-n_2} \cdots (q)_{n_{k-1}-n_k} (q)_{2n_k+1}} \label{iterateAndrews}\\
&\qquad\qquad\qquad\qquad
=\frac{(-q^{2k+1}/z, -zq^{2k+1}, q^{2k+1}; q^{2k+1})_\infty}{(q)_\infty}.  \nonumber
\end{align}
\end{theorem}

When $k=1$ in Theorem \ref{miterateAndrews}, it reduces to \eqref{nonsymbbailey301}.
If we set  $z=-q^{t}$ or equivalently, $z=-q^{-t}$, it leads to another form of Andrews--Gordon identity \eqref{eq:AGri}
\begin{align*}
\sum_{n_1\geq\cdots\geq n_k\geq 0} \frac{(q^{1-t})_{n_k} \ (q^{t})_{n_k +1} \ q^{\sum_{i=1}^{k} n_i(n_i+1)}} {(q)_{n_1-n_2} \cdots (q)_{n_{k-1}-n_k} (q)_{2n_k+1}}
=\frac{(q^t, q^{2k-t+1}, q^{2k+1}; q^{2k+1})_\infty}{(q)_\infty},
\end{align*}
combing the left hands of these two Andrews--Gordon identities, we can get
\begin{align*}
&\sum_{n_1\geq\cdots\geq n_k\geq 0} \frac{( q^{1-t})_{n_k} \ (q^{t})_{n_k +1} \ q^{\sum_{i=1}^{k} n_i(n_i+1)}} {(q)_{n_1-n_2} \cdots (q)_{n_{k-1}-n_k} (q)_{2n_k+1}}\\
&\qquad\qquad\qquad\qquad
=\sum_{n_1\geq\cdots\geq n_{k-1}\geq0} \frac{q^{n_1^2 +\cdots +n_{k-1}^2 +n_{t}+ \cdots+n_{k-1}}}{(q)_{n_1-n_2} \cdots(q)_{n_{k-2}-n_{k-1}} (q)_{n_{k-1}}}.
\end{align*}

When $t=0$, it is equivalent to identity \eqref{iterateAndrews3} as given in the following Corollary.

By substituting $z=e^{i\theta}$ with $\theta=0, \frac{\pi}{2}, \pi, \frac{\pi}{3} , \frac{2\pi}{3}$ into Theorem \ref{miterateAndrews}, and applying Jacobi's triple product identity \eqref{JTP}, we obtain the following identities.

\begin{corollary}
For $k\in \mathbb{N^+}$, we have
\begin{subequations}
\begin{align}
&\sum_{n_1\geq\cdots\geq n_k\geq 0} \frac{(-q)_{n_k}^2 \ q^{\sum_{i=1}^{k} n_i(n_i+1)}}
{(q)_{n_1-n_2}\cdots(q)_{n_{k-1}-n_k}(q)_{2n_k+1}}
=\frac{(-q^{2k+1}, -q^{2k+1}, q^{2k+1}; q^{2k+1})_\infty}{(q)_\infty},\label{iterateAndrews1}\\
&\sum_{n_1\geq\cdots\geq n_k\geq 0} \frac{(-q^2; q^2)_{n_k} \ q^{\sum_{i=1}^{k} n_i(n_i+1)}}
{(q)_{n_1-n_2} \cdots (q)_{n_{k-1}-n_k} (q)_{2n_k+1}}
=\frac{(q^{2k+1}, q^{6k+3}, q^{8k+4}; q^{8k+4})_\infty}{(q)_\infty},\label{iterateAndrews2}\\
&\sum_{n_1\geq\cdots\geq n_k\geq 0} \frac{(q)_{n_k} \ q^{\sum_{i=1}^{k} n_i(n_i+1)}}
{(q)_{n_1-n_2} \cdots (q)_{n_{k-1}-n_k} (q^{n_k+1})_{n_k+1}}
=\frac{(q^{2k+1};q^{2k+1})_\infty^3}{(q)_\infty},\label{iterateAndrews3}\\
&\sum_{n_1\geq\cdots\geq n_k\geq 0} \frac{(q^3;q^3)_{n_k} \ q^{\sum_{i=1}^{k}n_i(n_i+1)}}
{(q)_{n_1-n_2} \cdots (q)_{n_{k-1}-n_k} (q)_{2n_k+1}(q)_{n_k}}
=\frac{(q^{6k+3}; q^{6k+3})_\infty}{(q)_\infty}, \label{iterateAndrews4}\\
&\sum_{n_1\geq\cdots\geq n_k\geq 0} \frac{(-q^3;q^3)_{n_k} \ q^{\sum_{i=1}^{k} n_i(n_i+1)}} {(q)_{n_1-n_2} \cdots (q)_{n_{k-1}-n_k} (q)_{2n_k+1}(-q)_{n_k}} \label{iterateAndrews5}\\
&\qquad\qquad\qquad\qquad\qquad
=\frac{(q^{2k+1}; q^{2k+1})_\infty(q^{2k+1}, q^{10k+5}; q^{12k+6})_\infty}{(q)_\infty}.\nonumber
\end{align}
\end{subequations}
\end{corollary}

The above identities unify many known identities, such as when $k=1$, \eqref{iterateAndrews1} reduces to Entry 4.2.12 in the Ramanujan's ``Lost" Notebook \cite{AB09}; \eqref{iterateAndrews2}  reduces to  Entry 4.2.13 in \cite{AB09}; \eqref{iterateAndrews3} simplifies to the first identity in \cite[Theorem 4.3]{And19} and also proved by Lovejoy in  \cite[(1.9)]{Lov041}; \eqref{iterateAndrews4} gives the identity in \cite[(1.1)]{And19}, which is ``Dyson's favorite identity"; for \eqref{iterateAndrews5}, it can be found as the third identity of Theorem 4.3 in \cite{And19}.

Furthermore, by substituting the asymmetric bilateral Bailey pair \eqref{asymmetricbb1} into \eqref{weakiterate2}, we obtain the following result.

\begin{theorem}\label{miterateYao}
For $k\in \mathbb{N^+}$, we have
\begin{align}
&\sum_{n_1\geq\cdots\geq n_k\geq 0} \frac{(-q/z,-zq,-q)_{n_k} \ q^{\sum_{i=1}^{k} \binom{n_i+1}{2}}}{(q)_{n_1-n_2} \cdots (q)_{n_{k-1}-n_k} (q)_{2n_k+1}}\label{iterateYao}\\
&\qquad\qquad\qquad
=\frac{(-q)_\infty (-q^{k+1}/z, -zq^{k+1}, q^{k+1};q^{k+1})_\infty}{(q)_\infty}. \nonumber
\end{align}
\end{theorem}

When $k=1$, it reduces to \eqref{nonsymbbailey302}. In fact, when we take $q \mapsto q^2$ and $z=-q^{t-2k-2}$ in Theorem \ref{miterateYao}, also equivalent to $z=-q^{2k+2-t}$, we get the following identity
\begin{align*}
&\sum_{n_1\geq\cdots\geq n_k\geq 0} \frac{(q^{t-2k},-q^2; q^2)_{n_k}\ (q^{2k+2-t}; q^2)_{n_k +1} \ q^{\sum_{i=1}^{k} n_i(n_i+1) }} {(q^2; q^2)_{n_1-n_2} \cdots (q^2; q^2)_{n_{k-1}-n_k} (q^2; q^2)_{2n_k+1}}\\
&\qquad\qquad\qquad\qquad\qquad\qquad
=\frac{(-q^2; q^2)_\infty (q^t, q^{2k+2-t}, q^{2k+2}; q^{2k+2})_\infty}{(q^2; q^2)_\infty},
\end{align*}
which right hand is equivalent to \eqref{eq:AGev} given by  Andrews, hence we can get another new identity for their left hands
\begin{align*}
&\sum_{n_1\geq\cdots\geq n_k\geq 0} \frac{(q^{t-2k},-q^2; q^2)_{n_k}\ (q^{2k+2-t}; q^2)_{n_k +1} \ q^{\sum_{i=1}^{k} n_i(n_i+1) }} {(q^2; q^2)_{n_1-n_2} \cdots (q^2; q^2)_{n_{k-1}-n_k} (q^2; q^2)_{2n_k+1}}\\
&\qquad\qquad
=\sum_{n_1 \geq \cdots \geq n_{k-1} \geq 0} \frac{q^{n_1^2+ \cdots + n_{k-1}^2 +n_1 -n_2 +n_3 -n_{t-2} +n_{t-1} +n_t + \cdots + n_{k-1}}}{(q^2;q^2)_{n_1-n_2} \cdots (q^2;q^2)_{n_{k-2} -n_{k-1}} (q^2;q^2)_{n_{k-1}}}.
\end{align*}

Similarly, by taking $z=e^{i\theta}$ and $\theta=0, \frac{\pi}{2}, \pi, \frac{\pi}{3}, \frac{2\pi}{3}$, into Theorem \ref{miterateYao},  we get the following Andrews--Gordon type identities, respectively.

\begin{corollary}\label{coroiterateYao}
For $k\in \mathbb{N^+}$, we have
\begin{subequations}
\begin{align*}
&\sum_{n_1\geq\cdots\geq n_k\geq 0} \frac{(-q)_{n_k}^3 \ q^{\sum_{i=1}^k \binom{n_i+1}{2}}}{(q)_{n_1-n_2} \cdots (q)_{n_{k-1}-n_k} (q)_{2n_k+1}}\\
&\qquad\qquad\qquad\qquad
=\frac{(-q)_\infty(-q^{k+1}, -q^{k+1}, q^{k+1}; q^{k+1})_\infty}{(q)_\infty},\\
&\sum_{n_1\geq\cdots\geq n_k\geq 0} \frac{(-q)_{n_k} (-q^2; q^2)_{n_k} \ q^{\sum_{i=1}^k \binom{n_i+1}{2}}} {(q)_{n_1-n_2} \cdots (q)_{n_{k-1}-n_k} (q)_{2n_k+1}} \\
&\qquad\qquad\qquad\qquad
=\frac{(-q)_\infty(q^{k+1}, q^{3k+3}, q^{4k+4}; q^{4k+4})_\infty}{(q)_\infty},\\
&\sum_{n_1\geq\cdots\geq n_k\geq 0} \frac{(q)_{n_k}\ q^{\sum_{i=1}^k \binom{n_i+1}{2}}} {(q)_{n_1-n_2} \cdots(q)_{n_{k-1}-n_k}(q;q^2)_{n_k+1}}
=\frac{(-q)_\infty(q^{k+1};q^{k+1})_\infty^3}{(q)_\infty},\nonumber\\
&\sum_{n_1\geq\cdots\geq n_k\geq 0}
\frac{(-q)_{n_k} (q^3; q^3)_{n_k} \ q^{\sum_{i=1}^k \binom{n_i+1}{2}}} {(q)_{n_1-n_2} \cdots (q)_{n_{k-1}-n_k}(q)_{2n_k+1}(q)_{n_k}}
=\frac{(-q)_\infty(q^{3k+3};q^{3k+3})_\infty}{(q)_\infty},\nonumber\\
&\sum_{n_1\geq\cdots\geq n_k\geq 0} \frac{(-q^3; q^3)_{n_k} \ q^{\sum_{i=1}^k \binom{n_i+1}{2}}}
{(q)_{n_1-n_2}\cdots(q)_{n_{k-1}-n_k}(q)_{2n_k+1}}\\
&\qquad\qquad\qquad\qquad
=\frac{(-q)_\infty(q^{k+1}; q^{k+1})_\infty (q^{k+1},q^{5k+5};q^{6k+6})_\infty}{(q)_\infty}.
\end{align*}
\end{subequations}
\end{corollary}

Specially, when $k=1$ in Corollary \ref{coroiterateYao}, these identities are reduced to some known results including \cite[Theorem 3.5]{Yao25} and \cite[(1.10)]{Lov041}.

By inserting the asymmetric bilateral Bailey pair \eqref{asymmetricbb1} into \eqref{weakiterate4}, it implies the following identity.

\begin{theorem}\label{miterateweak3}
For $k\in \mathbb{N^+}$, we have
\begin{align}
&\sum_{n_1\geq\cdots\geq n_k\geq 0} \frac{(-1)^{\sum_{i=1}^k n_i} \ (-1/z)_{n_k+1} \ (-zq)_{n_k}} {(q)_{n_1-n_2} \cdots (q)_{n_{k-1}-n_k} (q^{n_k+1})_{n_k+1}} q^{\sum_{i=1}^{k} \binom{n_i+1}{2}}\\
&\qquad=\sum_{n=0}^\infty(-1)^{kn}q^{(k+1)
\binom{n+1}{2}} \ (z^{n}+z^{-n-1}).\nonumber
\end{align}
\end{theorem}

Setting $k=1$, it reduces to Theorem \ref{nonsymbbailey31}. Taking $z=e^{i\theta}$ with $\theta=0, \frac{\pi}{2}, \frac{\pi}{3}, \frac{2\pi}{3}$ into Theorem \ref{miterateweak3} respectively, it yields the following results.

\begin{corollary}
For $k\in \mathbb{N^+}$, we have
\begin{align*}
&\sum_{n_1\geq\cdots\geq n_k\geq 0} \frac{(-1)^{\sum_{i=1}^kn_i} \ q^{\sum_{i=1}^k \binom{n_i+1}{2}}\ (-q)_{n_k}^2}
{(q)_{n_1-n_2} \cdots (q)_{n_{k-1}-n_k} (q^{n_k+1})_{n_k+1}}
=\sum_{n=0}^{\infty} (-1)^{kn}q^{(k+1) \binom{n+1}{2}},\\[2mm]
&\sum_{n_1\geq\ldots\geq n_k\geq 0} \frac{(-1)^{\sum_{i=1}^kn_i}\ q^{\sum_{i=1}^k \binom{n_i+1}{2}}\ (-q^2;q^2)_{n_k}}
{(q)_{n_1-n_2} \cdots (q)_{n_{k-1}-n_k} (q^{n_k+1})_{n_k+1}}\\
&\qquad\qquad\qquad
=\sum_{n=0}^{\infty} (-1)^{n} q^{(k+1)\binom{2n+1}{2}}
\big(1+(-1)^{k+1} q^{(k+1)(2n+1)}\big),\\[2mm]
&\sum_{n_1\geq\cdots\geq n_k\geq 0} \frac{(-1)^{\sum_{i=1}^kn_i}\ q^{\sum_{i=1}^k \binom{n_i+1}{2}}\ (q^3;q^3)_{n_k}}
{(q)_{n_1-n_2} \cdots (q)_{n_{k-1}-n_k} (q)_{2n_k+1}}\\
&\qquad\qquad\qquad
=\sum_{n=0}^{\infty} (-1)^{(k+1)n} q^{(k+1) \binom{3n+1}{2}} \big(1-q^{(k+1)(6n+3)}\big),\\[2mm]
&\sum_{n_1\geq\cdots\geq n_k\geq 0} \frac{(-1)^{\sum_{i=1}^k n_i}\ q^{\sum_{i=1}^k \binom{n_i+1}{2}}\ (-q^3;q^3)_{n_k}}
{(q)_{n_1-n_2} \cdots (q)_{n_{k-1}-n_k} (q^{n_k+1})_{n_k+1} (-q)_{n_k}}\\[2mm]
&\
=\sum_{n=0}^{\infty} (-1)^{kn} q^{(k+1) \binom{3n+1}{2}} \big(1-2(-1)^{k} q^{(k+1)(3n+1)} +q^{(k+1)(6n+3)}\big).
\end{align*}
\end{corollary}

\section{Identities on false theta functions and generalized Hecke--type series}

In this section, by fitting the asymmetric bilateral Bailey pair \eqref{asymmetricbb1} into   the bilateral Bailey lemmas given by Andrews and Warnaar \cite{AW07}, we derive identities on false theta functions and generalized Hecke--type series.

Recall that a series is of generalized Hecke--type if it has the following form
\[
\sum_{(n,j)\in D} (-1)^{H(n,j)} q^{Q(n,j)+L(n,j)},
\]
where $H$ and $L$ are linear forms, $Q$ is a quadratic form, and $D$ is some subset of $\mathbb{Z}\times \mathbb{Z}$. Moreover, if  $Q(n,j)\geq 0$ for any $(n,j)\in D$, it is called Hecke--type series, which have received extensive attention since the study of Jacobi and Hecke, see, for example, \cite{And841, HM14}.

Andrews and Warnaar \cite{AW07} gave the following asymmetric bilateral Bailey lemma for specific bilateral Bailey pair $(\alpha_n(1,q^2),\ \beta_n(1,q^2))$ such  that
\begin{align}
\beta_n=\sum_{r=-n-1}^n\frac{\alpha_r}{(q^2; q^2)_{n-r}(q^2; q^2)_{n+r+1}}.\label{asymmetric}
\end{align}

\begin{lemma}\label{thm11}
If  $(\alpha_n, \beta_n)$ satisfies \eqref{asymmetric}, then
\begin{align}
\sum_{n=0}^\infty\frac{(q^2; q^2)_{2n+1}q^n}{(-q)_{2n+2}}\beta_n=\sum_{n=0}^\infty q^{n(n+2)}\sum_{j=-n-1}^nq^{-j(j+1)}\alpha_j,\label{hecke1}\\
\sum_{n=0}^\infty(q)_{2n+1}q^n\beta_n=\sum_{n=0}^\infty q^{n(n+3)/2}\sum_{j=-\lfloor\frac{n}{2}\rfloor-1}^{\lfloor\frac{n}{2}\rfloor}
q^{-2j(j+1)}\alpha_j.\label{hecke2}
\end{align}
\end{lemma}

By substituting $q\mapsto q^2$ in the asymmetric bilateral Bailey pair \eqref{asymmetricbb1}, then inserting it into formula \eqref{hecke1}, we obtain the following identity.

\begin{theorem}\label{hecke11}
We have
\begin{align}\label{hecke-1}
\sum_{n=0}^\infty \frac{(-1/z;q^2)_{n+1} \ (-zq^2; q^2)_n}{(-q)_{2n+2}}q^n = \sum_{n=0}^\infty q^{n(n+2)} \sum_{j=-n-1}^n z^{j}.
\end{align}
\end{theorem}

Then taking $z=e^{i\theta}$ into Theorem \ref{hecke11} and then letting $\theta=0$, that is, $z=1$, the following result can be derived.

\begin{corollary}\label{falsetheta1}
We have
\begin{align}
\sum_{n=0}^\infty\frac{(-q^2; q^2)^2_n}{(-q)_{2n+2}}q^n=\sum_{n=0}^\infty (n+1)q^{n(n+2)}.
\end{align}
\end{corollary}

In fact, this identity also appears in \cite[(1.6)]{AW07}. From other special values of $\theta$, it leads to identities on false theta functions.

\begin{corollary}\label{falsetheta2}
We have
\begin{align}
\sum_{n=0}^\infty\frac{(-q^4; q^4)_n}{(-q)_{2n+2}}q^n =\sum_{n=0}^\infty q^{4n(4n+2)} (1-q^{16n+8}).
\end{align}
\end{corollary}

\begin{proof}
Substituting $\theta=\frac{\pi}{2}$, that is, $z=i$ into \eqref{hecke-1},  we have
\begin{align*}
\sum_{n=0}^\infty \frac{(-q^4; q^4)_n}{(-q)_{2n+2}}q^n =\frac{1}{1-i} \sum_{n=0}^\infty q^{n(n+2)} \sum_{j=-n-1}^n i^{j},
\end{align*}
By considering the residue classes of $n$ modulo $4$ and simplifying the terminal term, it  yields
\begin{align*}
&\sum_{j=-\lfloor\frac{n+1}{4}\rfloor}^{\lfloor\frac{n}{4}\rfloor}i^{4j}
+\sum_{j=-\lfloor\frac{n+2}{4}\rfloor}^{\lfloor\frac{n-1}{4}\rfloor}i^{4j+1}
+\sum_{j=-\lfloor\frac{n+3}{4}\rfloor}^{\lfloor\frac{n-2}{4}\rfloor}i^{4j+2}
+\sum_{j=-\lfloor\frac{n+4}{4}\rfloor}^{\lfloor\frac{n-3}{4}\rfloor}i^{4j+3}\\
&\qquad=\Big(\lfloor\frac{n}{4}\rfloor+\lfloor\frac{n+1}{4}\rfloor+1\Big)
-\Big(\lfloor\frac{n-2}{4}\rfloor+\lfloor\frac{n+3}{4}\rfloor+1\Big)\\ &\qquad\quad
+i\bigg(\Big(\lfloor\frac{n-1}{4}\rfloor+\lfloor\frac{n+2}{4}\rfloor+1\Big)
-\Big(\lfloor\frac{n-3}{4}\rfloor+\lfloor\frac{n+4}{4}\rfloor+1\Big)\bigg)\\
&\qquad=\left\{\begin{array}{ll}
     1-i, & \hbox{if $n\equiv 0 \pmod{4}$ ,} \\[6pt]
     -1+i, & \hbox{if $n\equiv2\pmod{4}$,} \\[6pt]
     0, & \hbox{if $n\equiv 1,3\pmod{4}$},
               \end{array}
             \right.
\end{align*}
hence the right hand can be expressed as
\begin{align*}
\sum_{n=0}^\infty q^{4n(4n+2)}-\sum_{n=0}^\infty q^{(4n+2)(4n+4)},
\end{align*}
which completes the proof.
\end{proof}

\begin{corollary}\label{falsetheta3}
We have
\begin{align}
\sum_{n=0}^\infty\frac{(q^6; q^6)_n}{(-q)_{2n+2}(q^2;q^2)_n} q^n =\sum_{n=0}^\infty (-1)^n q^{3n(3n+2)} (1+q^{6n+3}).
\end{align}
\end{corollary}

This identity can also be seen in \cite[example 348]{Chu10}.

\begin{proof}
For $\theta=\frac{\pi}{3}$ in Theorem \ref{hecke11}, we can get the left hand side by multiplying both sides with $\frac{1}{1+e^{-\frac{\pi}{3}i}}$. Then for the right hand side, it turns to
\begin{align*}
\frac{1}{1+e^{-\frac{\pi}{3}i}} \sum_{n=0}^\infty q^{n(n+2)} \sum_{j=-n-1}^n e^{\frac{\pi}{3}ij}.
\end{align*}
Considering  the residue class of $n$ modulo $6$ and simplifying, it implies that
\begin{align*}
\sum_{j=-n-1}^n e^{\frac{\pi}{3}ij}=&\left\{\begin{array}{ll}
  \frac{3}{2}-\frac{\sqrt{3}}{2}i,
 & \hbox{if $n\equiv 0,1 \pmod{6}$ ,} \\[6pt]
   -\frac{3}{2}+\frac{\sqrt{3}}{2}i,
  & \hbox{if $n\equiv 3,4 \pmod{6}$,} \\[6pt]
    0, & \hbox{if $n\equiv 2,5\pmod{6}$}.
               \end{array}
             \right.
\end{align*}
Thus the right hand  side equals
\begin{align*}
&\sum_{n=0}^\infty q^{6n(6n+2)}+\sum_{n=0}^\infty q^{6(n+1)(6n+3)}-\Big(\sum_{n=0}^\infty q^{(6n+3)(6n+5)}+\sum_{n=0}^\infty q^{(6n+4)(6n+6)}\Big)\\
&=\sum_{n=0}^\infty (-1)^n(1+q^{6n+3})q^{3n(3n+2)},
\end{align*}
and the corollary is proved.
\end{proof}

Similar to Corollary \ref{falsetheta3}, by taking $\theta=\frac{2\pi}{3}$ into Theorem \ref{hecke11}, we  get the following identity.

\begin{corollary}\label{falsetheta4}
We have
\begin{align}
\sum_{n=0}^\infty \frac{(-q^6; q^6)_n}{(-q)_{2n+2}(-q^2;q^2)_n} q^n =\sum_{n=0}^\infty q^{3n(3n+2)} (1-q^{6n+3}).
\end{align}
\end{corollary}

This identity can be found as the example 347 in \cite{Chu10}.

By comparing the right hand sides of Corollary \ref{falsetheta3} and Corollary \ref{falsetheta4},   we can see that they are equivalent by setting $q\mapsto -q$, hence from their left hand sides,  we obtain the following interesting relation
\begin{align*}
\sum_{n=0}^\infty\frac{(-1)^n(q^6; q^6)_n}{(q)_{2n+1}(-q^2;q^2)_{n+1}}q^n
=\sum_{n=0}^\infty\frac{(-q^6; q^6)_n}{(-q)_{2n+2}(-q^2;q^2)_n}q^n.
\end{align*}

Finally, taking the asymmetric bilateral Bailey pair \eqref{asymmetricbb1} into formula \eqref{hecke2} gives the following result.

\begin{theorem}\label{hecke21}
We have
\begin{align}
\sum_{n=0}^\infty \frac{(-1/z;q^2)_{n+1} (-zq^2; q^2)_n}{(-q)_{2n+1}} q^n
=\sum_{n=0}^\infty q^{n(n+3)/2} \sum_{j=-\lfloor\frac{n}{2}\rfloor-1}^ {\lfloor\frac{n}{2}\rfloor} z^{j} q^{-j(j+1)}.
\end{align}
\end{theorem}

When $z=e^{i\theta}$ and substituting $\theta=0, \frac{\pi}{2}, \frac{\pi}{3},\frac{2\pi}{3}$ into Theorem \ref{hecke21}, respectively, we get the following identities on generalized Hecke--type series.

\begin{corollary}
We have
\begin{subequations}
\begin{align*}
&2\sum_{n=0}^\infty\frac{(-q^2; q^2)^2_n}{(-q)_{2n+1}}q^n
=\sum_{n=0}^\infty q^{n(n+3)/2}\sum_{j=-\lfloor\frac{n+2}{2}\rfloor}^{\lfloor\frac{n}{2}\rfloor}q^{-j(j+1)},\\
&\sum_{n=0}^\infty\frac{(-q^4;q^4)_n}{(-q)_{2n+1}}q^n
=\sum_{n=0}^\infty q^{n(n+3)/2}\sum_{j=-\lfloor\frac{n+2}{4}\rfloor}^{\lfloor\frac{n}{4}\rfloor}
(-1)^j q^{-2j(2j+1)},\\
&\sum_{n=0}^\infty\frac{(q^6; q^6)_n}{(-q)_{2n+1}(q^2;q^2)_n}q^n\\
&\quad =\sum_{n=0}^\infty q^{n(n+3)/2} \Big(\sum_{j=-\lfloor\frac{n+2}{12}\rfloor}^{\lfloor\frac{n}{12}\rfloor}
q^{-6j(6j+1)}-\sum_{j=-\lfloor\frac{n+6}{12}\rfloor}^{\lfloor\frac{n-4}{12}\rfloor}
q^{-(6j+2)(6j+3)} \Big),\\
&\sum_{n=0}^\infty\frac{(-q^6; q^6)_n}{(-q)_{2n+1}(-q^2;q^2)_n}q^n \\
&\quad =\sum_{n=0}^\infty q^{n(n+3)/2} \Big(\sum_{j=-\lfloor\frac{n+2}{6}\rfloor}^{\lfloor\frac{n}{6}\rfloor}
q^{-3j(3j+1)}-\sum_{j=-\lfloor\frac{n+4}{6}\rfloor}^{\lfloor\frac{n-2}{6}\rfloor}
q^{-(3j+1)(3j+2)} \Big).
\end{align*}
\end{subequations}
\end{corollary}

\vskip 15pt
\noindent {\small {\bf Acknowledgments.}
	This work is supported by the National Natural Science Foundation of China (No. 12071235) and the Fundamental Research Funds for the Central Universities of China.
	
\vskip 8pt 

\end{document}